\documentclass[11pt]{amsart}

\usepackage{amsfonts, amstext, amsmath, amsthm, amscd, amssymb} 
\usepackage{mathtools}

\usepackage{float} 
\usepackage{graphics}
\usepackage{caption}
\usepackage{subcaption}
\usepackage{import} 
\usepackage[dvipsnames]{xcolor}
\definecolor{MyCyan}{HTML}{00F9DE}

\usepackage{microtype}

\usepackage[hidelinks,pagebackref,pdftex]{hyperref}

\usepackage{marginnote}
\long\def\@savemarbox#1#2{\global\setbox#1\vtop{\hsize\marginparwidth 
  \@parboxrestore\tiny\raggedright #2}}
\newcommand{\RR}{\mathbb{R}}  
\newcommand{\ZZ}{\mathbb{Z}}  
\newcommand{\TT}{\mathbb{T}}
\newcommand{\calP}{\mathcal{P}}
\renewcommand{\SS}{\mathbb{S}}
\newcommand{\bdy}{\partial}
\newcommand{\Lk}{\mathrm{Lk}}
\renewcommand{\setminus}{{\smallsetminus}}

\newcommand{\from}{\colon\thinspace} 

\newtheorem{theorem}{Theorem}[section]
\newtheorem*{theorem*}{Theorem}
\newtheorem{proposition}[theorem]{Proposition}
\newtheorem{lemma}[theorem]{Lemma}

\newtheorem*{namedtheorem}{\theoremname}
\newcommand{\theoremname}{testing}
\newenvironment{named}[1]{\renewcommand{\theoremname}{#1}\begin{namedtheorem}}{\end{namedtheorem}}

\theoremstyle{definition}
\newtheorem{definition}[theorem]{Definition}

\newtheorem{remark}[theorem]{Remark}

\newcommand{\refthm}[1]{Theorem~\ref{Thm:#1}}
\newcommand{\reflem}[1]{Lemma~\ref{Lem:#1}}
\newcommand{\refprop}[1]{Proposition~\ref{Prop:#1}}

\newcommand{\refeqn}[1]{\eqref{Eqn:#1}}

\newcommand{\refdef}[1]{Definition~\ref{Def:#1}}
\newcommand{\refsec}[1]{Section~\ref{Sec:#1}}
\newcommand{\reffig}[1]{Figure~\ref{Fig:#1}}

\title[On the geometry of rod packings in the 3-torus]{On the geometry of rod packings \\in the 3-torus} 

\author{Connie On Yu Hui}
\address[]{School of Mathematics, Monash University, VIC 3800, Australia }
\email[]{onyu.hui@monash.edu}

\author{Jessica S. Purcell}
\address[]{School of Mathematics, Monash University, VIC 3800, Australia }
\email[]{jessica.purcell@monash.edu}

\subjclass[2020]{57K10, 57K32, 57K35 (primary); 57Z15 (secondary).}

\begin{document} 

\begin{abstract} 
Rod packings in the 3-torus encode information of some crystal structures in crystallography. They can be viewed as links in the $3$-torus, and tools from 3-manifold geometry and topology can be used to study their complements. In this paper, we initiate the use of geometrisation to study such packings. We analyse the geometric structures of the complements of simple rod packings, and find families that are hyperbolic and Seifert fibred. 
\end{abstract} 

\maketitle

\section{Introduction} 

One motivation for studying knots and links in the $3$-torus comes from crystallography, which  is a branch of chemistry and materials science that studies the structure and properties of crystalline materials.
A unit cell in crystallography describes a repeating pattern of particles in crystals, with repetitions occurring as translations along axes in $\RR^3$. The $3$-torus $\TT^3$ can be viewed as a $3$-dimensional cube with opposite faces glued; its universal cover is obtained by translations of the unit cell in three dimensions. Thus we regard the unit cell of a crystal as a $3$-torus, with particles in the unit cell depicted as points or solid spheres embedded in the $3$-torus. A 3-dimensional crystal structure is said to be 3-periodic, meaning it exhibits translational symmetry in three dimensions. 

In 1977, O'Keeffe and Andersson~\cite{OKeeffe-Andersson:RodPCrystalChem} observed that many crystal structures, including some common ones that have resisted other descriptions, can be more simply described in terms of what they call rod packings. Roughly, a rod packing is a packing of uniform cylinders, which represent linear or zigzag chains of atoms or chains of connected polyhedra. The paper describes some of the simpler cases of rod packings and their utility in describing crystal structures, which include what they call the cubic rod packing. In 2001, O'Keeffe \emph{et~al}~\cite{OKeeffeEtAl:CubicRodPackings} provided more results about cubic rod packings, classifying some of the simplest such structures in terms of arrangements in Euclidean space (the invariant ones; see \refsec{RodPackingsCrystallography}).
However, the classification of more complicated rod packings has not yet been fully explored.  
In this paper, using the perspective of links in the $3$-torus, we begin a classification of the geometry of complements of rod packings.

Various knot invariants have been developed to distinguish knots embedded in the 3-dimensional sphere or Euclidean space. Early knot theorists used tools such as crossing numbers to classify knots, which rely on the fact that knots in $3$-dimensional Euclidean space project onto a $2$-dimensional plane.  However, when the surrounding space is the $3$-torus instead of $\mathbb{R}^3$ or $\SS^3$, it is not always obvious how to apply classical tools. Embedded circles in the $3$-torus may correspond to nontrivial elements in the fundamental group of the $3$-torus.  Even the notion of a link diagram is not straightforward to define in this setting.

Here, we use 3-dimensional geometries to investigate the complements of rods in the $3$-torus. Determining the geometric decomposition does not require a link diagram or projection; this is one reason we introduce these methods to the classification of rod packings. 

Thurston's work on the geometry of Haken manifolds implies that the complement of each link in $\TT^3$ decomposes along essential spheres, discs, tori, and annuli into pieces admitting one of eight 3-dimensional geometric structures, including hyperbolic, sol, and six Seifert fibred structures~\cite{Thurston:3MKleinianGroupsHG}. 

In this paper, we consider the invariant cubic rod packings of O'Keeffe~\emph{et~al} of~\cite{OKeeffeEtAl:CubicRodPackings}, and show directly that five of them are hyperbolic in \refthm{Crystallography}. This gives examples of geometric structure for known rod packings.

We also classify the geometric type of complements of one or two rods in $\TT^3$; these are not hyperbolic. 

\begin{named}{\refthm{MainSingleRod_TwoRods}}
Let $R_1$ and $R_2$ be rod-shaped circles embedded in $\TT^3$.
\begin{itemize}
\item In the case of one rod, the $3$-manifold $\TT^3\setminus R_1$ is Seifert fibred. 

\item If $R_1$ and $R_2$ lift to be parallel to linearly dependent vectors in $\RR^3$, then $\TT^3\setminus(R_1\cup R_2)$ is Seifert fibred.
\item If $R_1$ and $R_2$ lift to linearly independent vectors, then $\TT^3\setminus(R_1\cup R_2)$ is toroidal.
\end{itemize}
In any case, when there are one or two rods, the complement is not hyperbolic. 
\end{named}

It is known that the 3-torus $\TT^3$ is obtained by Dehn filling the Borromean rings in $\SS^3$. The images of the Dehn filling solid tori in $\TT^3$ become three rods, forming the simplest of the cubic rod packings investigated by O'Keeffe~\emph{et~al}. We call these the \emph{standard rods}, and they are parallel to vectors $(1,0,0)$, $(0,1,0)$, and $(0,0,1)$, respectively. Thus it is natural to add to this rod packing a single additional rod, and to consider the resulting geometry. Such rod packings continue to have cubic lattice symmetry in crystallography. In this setting, we prove the following result. 

\begin{named}{\refthm{MainCharacterizeGeometryR}}
Let $a$, $b$, $c$ be integers such that $\mathrm{gcd}(a,b,c)=1$, and let $R$ be a rod parallel to the vector $(a,b,c)$ in $\TT^3$. Let $R_x$, $R_y$, $R_z$ denote the three standard rods in $\TT^3$. 
Then the 3-manifold $\TT^3\setminus(R_x \cup R_y \cup R_z \cup R)$ admits a complete hyperbolic structure if and only if $(a,b,c) \notin \{(\pm 1, 0, 0), (0,\pm 1, 0), (0, 0, \pm 1)\}$. 
\end{named}

Theorems~\ref{Thm:Crystallography}, \ref{Thm:MainSingleRod_TwoRods}, and~\ref{Thm:MainCharacterizeGeometryR} concern 3-periodic links. 
There are other recent results related to $2$-periodic links. This work includes a tabulation of knots in the thickened torus due to Akimova and Matveev~\cite{Akimova-Matveev:VirtualKnots}. The geometry of the complements of alternating 2-periodic links was considered by Champanerkar, Kofman, and Purcell~\cite{CKP:Biperiodic}, and by Adams~\emph{et~al}~\cite{AdamsEtAl:Tilings}.

Other results on $3$-periodic links are closely related to materials science and biological science. Rosi \emph{et al} present the structures of certain metal-organic frameworks (MOFs), which have useful gas and liquid adsorption properties, in the context of the underlying rod packing structures \cite{RosiEtAl:RodPackingsMOF}. Evans, Robins, and Hyde construct a collection of general rod packings using triply periodic minimal surfaces in Euclidean space and the construction method is potentially significant for the design of materials with unusual physical properties \cite{ERH:PeriodicEntanglementII}.   
Rod packing structures in this paper also appear in biological science, in particular, in the structure of skin. 
Norl{\'e}n and Al-Amoudi suggest that the outermost layer of mammalian skin is made up of filaments 
arranged according to a cubic-like rod-packing symmetry~\cite{Norlen-AA:CubicRodAndSkin}.  Evans and Hyde similarly propose that the keratin fibres in the outermost layer of skin form a $3$-dimensional weaving related to the $^+\Sigma$ rod packing discussed below~\cite{EH:SwollenCorneocytes}.  

Others studying geometry and periodic links include Evans, Robins, and Hyde, who arrange 3-periodic links by minimising energy functions~\cite{ERH:IdealGeometry}. Evans and Schr\"oder-Turk use 2-dimensional hyperbolic geometry to form a neighbourhood of triply periodic links embedded in $\RR^3$ or the 3-torus~\cite{ES:HyperbolicBiological}.

\subsection{Organization}  
\refsec{Prelim} provides preliminaries. In \refsec{RodPackingsCrystallography}, we consider five explicit examples arising from \cite{OKeeffeEtAl:CubicRodPackings}. We show that they can be reconsidered as links in the 3-sphere, and use that to prove \refthm{Crystallography}. We prove \refthm{MainSingleRod_TwoRods} in \refsec{OneTwoRods}, focusing on one or two arbitrary rods in the $3$-torus.  Section~\ref{Sec:StdRodsPlusRod} gives an argument for \refthm{MainCharacterizeGeometryR}. 

\subsection{Acknowledgements}
We thank Norman Do for the helpful discussions. Parts of this work were inspired by the Monash honours thesis of Sargon Al-Jeloo~\cite{Al-Jeloo:HonoursThesis}. The research was partially supported by the Australian Research Council, grant DP210103136. We also thank the anonymous referee for their helpful comments, particularly around crystallographic notions.

\section{Preliminaries}\label{Sec:Prelim} 

Throughout this paper, we regard the \emph{$3$-torus} as the unit cube $[0,1]^3$ in $\mathbb{R}^3$ with opposite faces glued in the standard way unless otherwise specified. Let $L$ be a $1$-dimensional or $2$-dimensional submanifold in a $3$-manifold $M$. The notation $N(L)$ denotes an open tubular neighbourhood of $L$ in $M$; $\overline{N}(L)$ denotes a closed tubular neighbourhood.

\begin{definition} \label{Def:Rodsss}
Denote by $\mathcal{P}\from \RR^3  \to \TT^3$
the covering map $\mathcal{P}(x,y,z) \coloneqq ([x-\lfloor x \rfloor], [y - \lfloor y \rfloor], [z - \lfloor z \rfloor])$.  

\vspace{2mm}

\noindent{(1)} A \emph{rod} $R$ in the $3$-torus is the projection $\mathcal{P}(L)$ of a Euclidean straight line $L$ in $\RR^3$. If the set $R=\mathcal{P}(L)$ is homeomorphic to a circle embedded in $\TT^3$, we call $R$ a \emph{rod-shaped circle}, or simply a \emph{rod} when the context is clear. 

\vspace{1mm}

\noindent{(2)} A \emph{lift of a rod} $R$ is a connected component of the pre-image set $\calP^{-1}(R)$ under the projection map $\calP\from \RR^3 \to \TT^3$. 

\vspace{1mm}

\noindent{(3)} Let $a$, $b$, $c$ be real numbers, not all zero. Denote by $L$ a straight line in $\RR^3$ that has the same direction as the vector $(a,b,c)$. We call the rod $R=\calP(L)$ an \emph{$(a,b,c)$-rod}. Unless otherwise specified, each $(a,b,c)$-rod is associated with the map $R_e \from [0,1] \to \TT^3$ defined as $R_e \coloneqq \calP \circ R_{\ell}$, where $R_{\ell} \from [0,1] \to \RR^3$ is the linear map $R_{\ell}(t) = (a_0, b_0, c_0) + t(a,b,c)$ for some $(a_0, b_0, c_0) \in L$.

\vspace{1mm}

\noindent{(4)} We denote by $R_x$, $R_y$, and $R_z$ the $(1,0,0)$-rod, $(0,1,0)$-rod, and $(0,0,1)$-rod respectively. These are shown in \reffig{Homeomorphic}, left. 
  We call $R_x$, $R_y$, and $R_z$ the \emph{standard rods} in $\TT^3$. 
\end{definition} 

\begin{figure}
\includegraphics{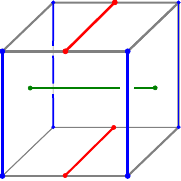} 
\hspace{2.5cm}
\includegraphics{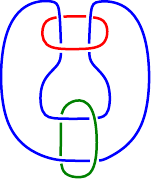} 
\caption{$\TT^3 \setminus (R_x\cup R_y\cup R_z)$ and $\SS^3 \setminus (C_x\cup C_y\cup C_z)$.}
\label{Fig:Homeomorphic}
\end{figure}

Note in \reffig{Homeomorphic} left that we have placed the rods on the unit cube such that they have lifts in $\RR^3$ given by the lines
\begin{equation}\label{Eqn:RodFormulas}
  R_x : (0, \tfrac{1}{2},0)+(1,0,0)t; \;
  R_y : (\tfrac{1}{2}, 0, \tfrac{1}{2}) + (0,1,0)t; \;
  R_z : (0,0,0) + (0,0,1)t
\end{equation} 
The complement of these rods is a 3-manifold well-defined up to ambient isotopy of the rods. However, we will typically illustrate them as shown here.

Further notation: Let $p$, $q$ be integers, not both zero. 
A \emph{$(p,q)$-curve around a rod-shaped circle $R$} in $\TT^3$ is a $(p,q)$-curve in the torus boundary of a neighbourhood of $R$, where $\partial \overline{N}(R)$ is framed such that a meridian, i.e.\ the boundary of an essential disc in $N(R)\subset \TT^3$, is the $(1,0)$-curve. For example, the curve $\mu_y$ in \reffig{OctaDecompT3Mu} (left) is a $(1,0)$-curve around the green rod $R_y$.  

\begin{lemma}  
Let $(a,b,c)\in \ZZ^3\setminus\{(0,0,0)\}$. The map $R_e\from [0,1]\to \TT^3$ associated to the $(a,b,c)$-rod $R$ represents a simple closed curve in $\TT^3$ if and only if $\mathrm{gcd}(a,b,c)=1$. 
\end{lemma}  

\begin{proof}
Since any line $L$ that has the same direction as the vector $(a,b,c)$ can be translated to run through $(0,0,0)$, we may assume $L$ is a line in $\RR^3$ through $(0,0,0)$ and $(a,b,c)$. Consider the $(a,b,c)$-rod $R$ as $\calP(L)$. Note that $\calP(L)$ is a closed loop because $(a,b,c)\in\ZZ^3$, and indeed the closed segment of $L$ with endpoints $(0,0,0)$ and $(a,b,c)$ projects to a loop in $\TT^3$. 

Suppose $R_e\from [0,1]\to\TT^3$ is a simple loop. Then the interior of the line segment between $(0,0,0)$ and $(a,b,c)$ does not intersect any points with all three coordinates being integers, for such a point would project to the same image as $(0,0,0)$ and $(a,b,c)$ in $\TT^3$. Thus $\mathrm{gcd}(a,b,c) = 1$. 

Suppose $R_e\from [0,1]\to\TT^3$ is not simple. Then there would exist points $p \neq q$ in the interior of the segment of $L$ between $(0,0,0)$ and $(a,b,c)$ such that $\mathcal{P}(p) = \mathcal{P}(q)$. Without loss of generality, assume $|p| < |q|$. Let $t$ be a parameter that increases from zero to $|p|$, and let $\vec{u}$ be the unit vector $(a,b,c)/\lVert(a,b,c)\rVert$. Note that $\mathcal{P}(p-t\vec{u}) = \mathcal{P}(q-t\vec{u})$ for each $t\in [0,|p|]$. So $\mathcal{P}(\vec{0}) = \mathcal{P}(q-|p|\vec{u})$. Since $q-|p|\vec{u}$ is neither $(0,0,0)$ nor $(a,b,c)$ and $\mathcal{P}(q-|p|\vec{u}) = \mathcal{P}(\vec{0})$, we have $(q-|p|\vec{u})$ equals some point $(m,n,r)$ in the interior of the segment of $L$ between $(0,0,0)$ and $(a,b,c)$, with $m,n,r\in\mathbb{Z}$. Therefore, $(a,b,c)=k(m,n,r)$ for some $k\in\mathbb{Z}\setminus\{-1,0,1\}$, and thus $|k|$ would be a common divisor of $a$, $b$, and $c$ that is greater than one. 
\end{proof}

\subsection{Decompositions into ideal octahedra}
In this section, we review a decomposition of the complement of the Borromean rings in $\SS^3$ into ideal octahedra. W.~Thurston may have been the first to observe such a decomposition exists, in~\cite[Chapter~3]{Thurston:GeomTop3Manifolds}.
We recall a slightly different decomposition following the method explained in~\cite[Chapter~7]{Purcell:HyperbolicKnotTheory}, which originally appeared in 
Agol and D.~Thurston's appendix in~\cite{Lackenby:AltVol}.
We also review a decomposition of the complement of the three standard rods in $\TT^3$ into ideal octahedra. This decomposition is known to experts. We give a geometric proof here, which also provides a geometric proof that this manifold is homeomorphic to the Borromean rings complement in $\SS^3$. While the homeomorphism is well known, the geometric proof will be important for the arguments of this paper, and so we include it here.

\begin{lemma}\label{Lem:BorromRingsDecomp}
Let $C_x$, $C_y$, and $C_z$ denote the link components of the Borromean rings in $\SS^3$. The link complement $\SS^3\setminus (C_x\cup C_y\cup C_z)$ can be decomposed into two ideal octahedra. 
\end{lemma}

\begin{proof}
The proof is geometric. Arrange $\SS^3\setminus (C_x\cup C_y\cup C_z)$ with $C_z$  in the projection plane, and $C_x$ and $C_y$ orthogonal to the projection plane, bounding disjoint discs meeting the first component twice; see \reffig{Homeomorphic}, right. 

\begin{figure}
\includegraphics{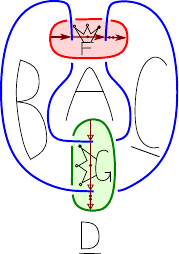}
\quad \quad \quad 
\includegraphics{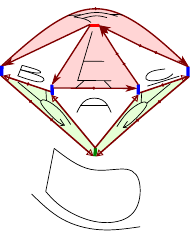}
\quad \quad \quad 
\includegraphics{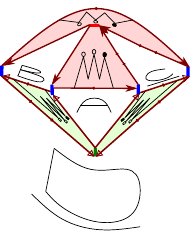}
\caption{Octahedral decomposition of $\SS^3 \setminus (C_x\cup C_y\cup C_z)$. Left: Link diagram with six edges and eight faces. Middle: Octahedron above link diagram (viewed from interior). Right: Octahedron below link diagram (viewed from exterior).  \label{Fig:OctDecompositionS3}}
\end{figure}  

As shown on the left of \reffig{OctDecompositionS3}, we label regions of the projection plane with letters $A$, $B$, $\underline{C}$, and $\underline{D}$. We shade the discs bounded by $C_x$ and $C_y$ (pink and green in the figure, respectively), and label them with letters $F$, $G$, and crowns. 

Cut along the surface of the projection plane, then cut along the two shaded 2-punctured discs. This splits the link complement into two balls with remnants of the link on their boundaries. Each has white faces $A$, $B$, $\underline{C}$, $\underline{D}$, coming from the projection plane, and shaded faces labelled with $F$, $G$ or crowns, coming from the 2-punctured discs. The faces intersect in six ideal edges, coloured brown in the figure. When we shrink the remnants of the link to ideal vertices, we obtain two ideal octahedra, shown in the middle and right of \reffig{OctDecompositionS3}. Note that the faces are checkerboard coloured. 

Gluing the ideal octahedra reverses this cutting procedure. Briefly, two opposite ideal triangles at the top of each octahedron are shaded; these are glued together by folding across the ideal vertex at the top of the octahedron. Similarly, two opposite shaded ideal triangles on the bottom are glued by folding in the same octahedron. The white faces are identified to the corresponding white face in the opposite octahedron by the identity map. 
\end{proof}

\begin{lemma}\label{Lem:StdRodsDecomp}
Let $R_x$, $R_y$, and $R_z$ denote the standard rods in $\TT^3$ as defined in \refdef{Rodsss}(4). The complement $\TT^3\setminus (R_x\cup R_y\cup R_z)$ can be decomposed into two ideal octahedra.

Moreover, the octahedra and its face pairings are identical to those of the Borromean rings, of \reflem{BorromRingsDecomp}. Thus there is a homeomorphism
\[ h \from \TT^3\setminus (R_x\cup R_y\cup R_z) \to \SS^3\setminus (C_x\cup C_y\cup C_z). \]
\end{lemma}

\begin{proof}
View $\TT^3$ as the unit cube with opposite faces identified, and with rods positioned as in \refeqn{RodFormulas}. Slice the cube horizontally along the plane $\{z=1/2\}$; this plane contains the rod $R_y$.

The result is two rectangular boxes, with remnants of the link components at the four vertical edges on the boundary, and across the top and bottom of each box. Observe that shrinking these remnants of the link to ideal vertices yields an ideal octahedron.

Colour white the top and bottom faces of each rectangular box. Label the top faces of the top box $\underline{C}$ and $B$ as shown in \reffig{OctDecompT3}, the middle faces (bottom of the top box and top of the bottom box) $A$ and $\underline{D}$, and the bottom faces of the bottom box $\underline{C}$ and $B$. Then the white faces glue to the correspondingly labelled white face on the opposite octahedron.

\begin{figure}
\includegraphics{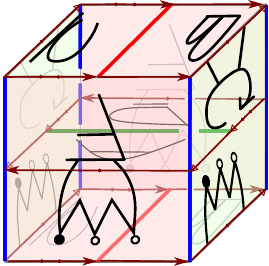}
\hspace{6mm} 
\includegraphics{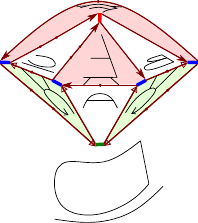}
\hspace{1mm}  
\includegraphics{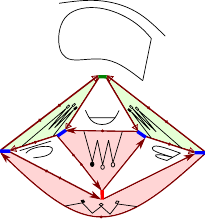}
\caption{Octahedral decomposition of  $\TT^3 \setminus (R_x\cup R_y\cup R_z)$. Left: $\TT^3 \setminus (R_x\cup R_y\cup R_z)$ with six edges and eight faces.  Middle: The upper ideal octahedron (view from exterior).  Right: The lower ideal octahedron (view from exterior)} \label{Fig:OctDecompT3}
\end{figure}  

Shade the front and back of each rectangular box pink, and label the top one $F$ and the bottom with a crown. Shade the left and right of each rectangular box green, and label the top one $G$ and the bottom one with a crown. Note that shaded faces are glued in pairs in the same octahedron. 

With this labeling, the ideal octahedra are checkerboard coloured, and the gluing of faces is identical to that of the Borromean rings complement. Moreover, there is a homeomorphism taking the pair of octahedra in the one link complement to that of the other; to go from  \reffig{OctDecompT3} to \reffig{OctDecompositionS3}, the homeomorphism moves the viewpoint to the interior of the octahedron in the middle, and rotates the octahedron on the right. Since gluings and octahedra agree, this is a homeomorphism of spaces. 
\end{proof}

Note that in the octahedral decompositions, neighbourhoods of the ideal vertices corresponding to rods $R_x$, $R_y$, and $R_z$ are mapped to neighbourhoods of ideal vertices corresponding to the link components $C_x$, $C_y$, and $C_z$, respectively, of the Borromean rings in $\SS^3$.

The following results are almost immediate consequences of \reflem{StdRodsDecomp}.

\begin{theorem} \label{Thm:ThreeStandardRods}  
The complement of the three standard rods in the $3$-torus, denoted by $\TT^3\setminus (R_x \cup R_y \cup R_z)$, admits a complete hyperbolic structure.  
\end{theorem} 

\begin{proof}
By \reflem{StdRodsDecomp}, the $3$-manifold $\TT^3\setminus (R_x \cup R_y \cup R_z)$ is homeomorphic to the complement of the Borromean rings in the $3$-sphere, which is well-known to admit a complete hyperbolic structure; see \cite{Thurston:GeomTop3Manifolds}, or \cite[Section~7.2]{Purcell:HyperbolicKnotTheory}
\end{proof}

\begin{proposition} \label{Prop:010101DehnFilling}
The $3$-torus is homeomorphic to the Dehn filling of the Borromean rings in $\SS^3$ along slopes corresponding to homological longitudes, namely $(0,1)$, $(0,1)$ and $(0,1)$ in the standard framing in $\SS^3$. 
\end{proposition} 

\begin{proof}
We show that the meridian of a tubular neighbourhood of $R_y$ in $\TT^3$ is mapped to the longitude of a tubular neighbourhood of $C_y$ in $\SS^3$ under the homeomorphism $h$ of \reflem{StdRodsDecomp}. The arguments for $R_x$ and $C_x$ and for $R_z$ and $C_z$ are similar. 

Let $\mu_y$ be a meridian of a tubular neighbourhood of $R_y$ in $\TT^3$; see \reffig{OctaDecompT3Mu} left. In the decomposition of \reflem{StdRodsDecomp}, $\mu_y$ meets white faces labelled $A$ and $\underline{D}$ in both ideal octahedra, and runs parallel to a neighbourhood of the green ideal vertex in both. See \reffig{OctaDecompT3Mu} right.

\begin{figure}
\begingroup%
  \makeatletter%
  \providecommand\color[2][]{%
    \errmessage{(Inkscape) Color is used for the text in Inkscape, but the package 'color.sty' is not loaded}%
    \renewcommand\color[2][]{}%
  }%
  \providecommand\transparent[1]{%
    \errmessage{(Inkscape) Transparency is used (non-zero) for the text in Inkscape, but the package 'transparent.sty' is not loaded}%
    \renewcommand\transparent[1]{}%
  }%
  \providecommand\rotatebox[2]{#2}%
  \newcommand*\fsize{\dimexpr\f@size pt\relax}%
  \newcommand*\lineheight[1]{\fontsize{\fsize}{#1\fsize}\selectfont}%
  \ifx\svgwidth\undefined%
    \setlength{\unitlength}{86.96288109bp}%
    \ifx\svgscale\undefined%
      \relax%
    \else%
      \setlength{\unitlength}{\unitlength * \real{\svgscale}}%
    \fi%
  \else%
    \setlength{\unitlength}{\svgwidth}%
  \fi%
  \global\let\svgwidth\undefined%
  \global\let\svgscale\undefined%
  \makeatother%
  \begin{picture}(1,0.98915377)%
    \lineheight{1}%
    \setlength\tabcolsep{0pt}%
    \put(0,0){\includegraphics[width=\unitlength,page=1]{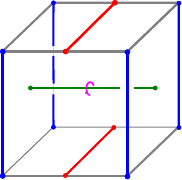}}%
    \put(0.47151894,0.58746338){\color[rgb]{1,0,1}\makebox(0,0)[lt]{\lineheight{1.25}\smash{\begin{tabular}[t]{l}$\mu_y$\end{tabular}}}}%
    \put(0,0){\includegraphics[width=\unitlength,page=2]{T3WithoutRxRyRz6edgesRotatedMu.pdf}}%
  \end{picture}%
\endgroup%

\hspace{6mm} 
\begingroup%
  \makeatletter%
  \providecommand\color[2][]{%
    \errmessage{(Inkscape) Color is used for the text in Inkscape, but the package 'color.sty' is not loaded}%
    \renewcommand\color[2][]{}%
  }%
  \providecommand\transparent[1]{%
    \errmessage{(Inkscape) Transparency is used (non-zero) for the text in Inkscape, but the package 'transparent.sty' is not loaded}%
    \renewcommand\transparent[1]{}%
  }%
  \providecommand\rotatebox[2]{#2}%
  \newcommand*\fsize{\dimexpr\f@size pt\relax}%
  \newcommand*\lineheight[1]{\fontsize{\fsize}{#1\fsize}\selectfont}%
  \ifx\svgwidth\undefined%
    \setlength{\unitlength}{95.0300501bp}%
    \ifx\svgscale\undefined%
      \relax%
    \else%
      \setlength{\unitlength}{\unitlength * \real{\svgscale}}%
    \fi%
  \else%
    \setlength{\unitlength}{\svgwidth}%
  \fi%
  \global\let\svgwidth\undefined%
  \global\let\svgscale\undefined%
  \makeatother%
  \begin{picture}(1,1.12514458)%
    \lineheight{1}%
    \setlength\tabcolsep{0pt}%
    \put(0,0){\includegraphics[width=\unitlength,page=1]{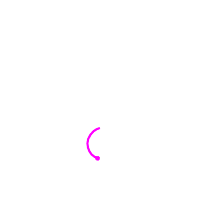}}%
    \put(0.00234968,0.35582909){\color[rgb]{1,0,1}\makebox(0,0)[lt]{\lineheight{1.25}\smash{\begin{tabular}[t]{l}$\textrm{1st half}$\end{tabular}}}}%
    \put(0.00207938,0.24488247){\color[rgb]{1,0,1}\makebox(0,0)[lt]{\lineheight{1.25}\smash{\begin{tabular}[t]{l}$\textrm{of } \mu_y$\end{tabular}}}}%
    \put(0,0){\includegraphics[width=\unitlength,page=2]{BorromeanRingsFrontOctahedronExtMu.pdf}}%
  \end{picture}%
\endgroup%

\hspace{3mm} 
\begingroup%
  \makeatletter%
  \providecommand\color[2][]{%
    \errmessage{(Inkscape) Color is used for the text in Inkscape, but the package 'color.sty' is not loaded}%
    \renewcommand\color[2][]{}%
  }%
  \providecommand\transparent[1]{%
    \errmessage{(Inkscape) Transparency is used (non-zero) for the text in Inkscape, but the package 'transparent.sty' is not loaded}%
    \renewcommand\transparent[1]{}%
  }%
  \providecommand\rotatebox[2]{#2}%
  \newcommand*\fsize{\dimexpr\f@size pt\relax}%
  \newcommand*\lineheight[1]{\fontsize{\fsize}{#1\fsize}\selectfont}%
  \ifx\svgwidth\undefined%
    \setlength{\unitlength}{98.40814203bp}%
    \ifx\svgscale\undefined%
      \relax%
    \else%
      \setlength{\unitlength}{\unitlength * \real{\svgscale}}%
    \fi%
  \else%
    \setlength{\unitlength}{\svgwidth}%
  \fi%
  \global\let\svgwidth\undefined%
  \global\let\svgscale\undefined%
  \makeatother%
  \begin{picture}(1,1.06508395)%
    \lineheight{1}%
    \setlength\tabcolsep{0pt}%
    \put(0,0){\includegraphics[width=\unitlength,page=1]{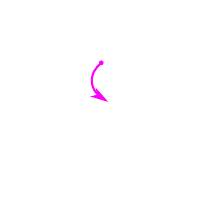}}%
    \put(0.02457111,0.68599309){\color[rgb]{1,0,1}\makebox(0,0)[lt]{\lineheight{1.25}\smash{\begin{tabular}[t]{l}$\textrm{2nd half}$\end{tabular}}}}%
    \put(0.02431009,0.57885497){\color[rgb]{1,0,1}\makebox(0,0)[lt]{\lineheight{1.25}\smash{\begin{tabular}[t]{l}$\textrm{of } \mu_y$\end{tabular}}}}%
    \put(0,0){\includegraphics[width=\unitlength,page=2]{BorromeanRingsBackOctahedronRotate1Mu.pdf}}%
  \end{picture}%
\endgroup%

\caption{Left: $\TT^3 \setminus (R_x\cup R_y\cup R_z)$ and $\mu_y$.  Middle: Upper octahedron and half of $\mu_y$ (view from exterior).  Right: Lower octahedron and second half of $\mu_y$ (view from exterior)
\label{Fig:OctaDecompT3Mu}}
\end{figure}  

In the complement of the Borromean rings, a curve running from $A$ to $\underline{D}$ parallel to the green ideal vertex in an octahedron is half of a longitude; see \reffig{OctaDecompS3Mu} right. The two arcs of $h(\mu_y)$ in the two octahedra therefore form a longitude of $C_y$. 
\begin{figure}
\begingroup%
  \makeatletter%
  \providecommand\color[2][]{%
    \errmessage{(Inkscape) Color is used for the text in Inkscape, but the package 'color.sty' is not loaded}%
    \renewcommand\color[2][]{}%
  }%
  \providecommand\transparent[1]{%
    \errmessage{(Inkscape) Transparency is used (non-zero) for the text in Inkscape, but the package 'transparent.sty' is not loaded}%
    \renewcommand\transparent[1]{}%
  }%
  \providecommand\rotatebox[2]{#2}%
  \newcommand*\fsize{\dimexpr\f@size pt\relax}%
  \newcommand*\lineheight[1]{\fontsize{\fsize}{#1\fsize}\selectfont}%
  \ifx\svgwidth\undefined%
    \setlength{\unitlength}{95.64486058bp}%
    \ifx\svgscale\undefined%
      \relax%
    \else%
      \setlength{\unitlength}{\unitlength * \real{\svgscale}}%
    \fi%
  \else%
    \setlength{\unitlength}{\svgwidth}%
  \fi%
  \global\let\svgwidth\undefined%
  \global\let\svgscale\undefined%
  \makeatother%
  \begin{picture}(1,1.15585071)%
    \lineheight{1}%
    \setlength\tabcolsep{0pt}%
    \put(0,0){\includegraphics[width=\unitlength,page=1]{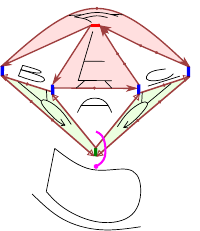}}%
    \put(0.6369865,0.42345655){\color[rgb]{1,0,1}\makebox(0,0)[lt]{\lineheight{1.25}\smash{\begin{tabular}[t]{l}$\textrm{1st half}$\end{tabular}}}}%
    \put(0.63671788,0.3132231){\color[rgb]{1,0,1}\makebox(0,0)[lt]{\lineheight{1.25}\smash{\begin{tabular}[t]{l}$\textrm{of } \mu_y$\end{tabular}}}}%
  \end{picture}%
\endgroup%

\quad \quad 
\begingroup%
  \makeatletter%
  \providecommand\color[2][]{%
    \errmessage{(Inkscape) Color is used for the text in Inkscape, but the package 'color.sty' is not loaded}%
    \renewcommand\color[2][]{}%
  }%
  \providecommand\transparent[1]{%
    \errmessage{(Inkscape) Transparency is used (non-zero) for the text in Inkscape, but the package 'transparent.sty' is not loaded}%
    \renewcommand\transparent[1]{}%
  }%
  \providecommand\rotatebox[2]{#2}%
  \newcommand*\fsize{\dimexpr\f@size pt\relax}%
  \newcommand*\lineheight[1]{\fontsize{\fsize}{#1\fsize}\selectfont}%
  \ifx\svgwidth\undefined%
    \setlength{\unitlength}{96.01021467bp}%
    \ifx\svgscale\undefined%
      \relax%
    \else%
      \setlength{\unitlength}{\unitlength * \real{\svgscale}}%
    \fi%
  \else%
    \setlength{\unitlength}{\svgwidth}%
  \fi%
  \global\let\svgwidth\undefined%
  \global\let\svgscale\undefined%
  \makeatother%
  \begin{picture}(1,1.15145229)%
    \lineheight{1}%
    \setlength\tabcolsep{0pt}%
    \put(0,0){\includegraphics[width=\unitlength,page=1]{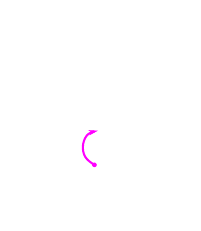}}%
    \put(0.6383679,0.42274594){\color[rgb]{1,0,1}\makebox(0,0)[lt]{\lineheight{1.25}\smash{\begin{tabular}[t]{l}$\textrm{2nd half}$\end{tabular}}}}%
    \put(0.6381003,0.31293196){\color[rgb]{1,0,1}\makebox(0,0)[lt]{\lineheight{1.25}\smash{\begin{tabular}[t]{l}$\textrm{of } \mu_y$\end{tabular}}}}%
    \put(0,0){\includegraphics[width=\unitlength,page=2]{BorromeanRingsBackOctahedronMu.pdf}}%
  \end{picture}%
\endgroup%

\quad \quad 
\begingroup%
  \makeatletter%
  \providecommand\color[2][]{%
    \errmessage{(Inkscape) Color is used for the text in Inkscape, but the package 'color.sty' is not loaded}%
    \renewcommand\color[2][]{}%
  }%
  \providecommand\transparent[1]{%
    \errmessage{(Inkscape) Transparency is used (non-zero) for the text in Inkscape, but the package 'transparent.sty' is not loaded}%
    \renewcommand\transparent[1]{}%
  }%
  \providecommand\rotatebox[2]{#2}%
  \newcommand*\fsize{\dimexpr\f@size pt\relax}%
  \newcommand*\lineheight[1]{\fontsize{\fsize}{#1\fsize}\selectfont}%
  \ifx\svgwidth\undefined%
    \setlength{\unitlength}{85.56467473bp}%
    \ifx\svgscale\undefined%
      \relax%
    \else%
      \setlength{\unitlength}{\unitlength * \real{\svgscale}}%
    \fi%
  \else%
    \setlength{\unitlength}{\svgwidth}%
  \fi%
  \global\let\svgwidth\undefined%
  \global\let\svgscale\undefined%
  \makeatother%
  \begin{picture}(1,1.4245337)%
    \lineheight{1}%
    \setlength\tabcolsep{0pt}%
    \put(0,0){\includegraphics[width=\unitlength,page=1]{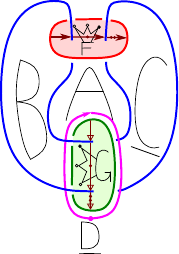}}%
    \put(0.67734486,0.22342607){\color[rgb]{1,0,1}\makebox(0,0)[lt]{\lineheight{1.25}\smash{\begin{tabular}[t]{l}$h(\mu_y)$\end{tabular}}}}%
  \end{picture}%
\endgroup%

\caption{Identifying $h(\mu_y)$ in $\SS^3 \setminus (C_x\cup C_y\cup C_z)$. Left: Octahedron above link diagram and half of $\mu_y$ (view from interior).  Middle: Octahedron below link diagram and half of $\mu_y$ (from exterior).  Right: $\SS^3 \setminus (C_x\cup C_y\cup C_z)$ and $h(\mu_y)$}
\label{Fig:OctaDecompS3Mu}
\end{figure}  
\end{proof}

\begin{lemma} \label{Lem:LinkingNum} 
If a rod-shaped circle $R$, embedded in $\TT^3\setminus(R_x \cup R_y \cup R_z)$, is parallel to the vector $(a,b,c) \in \mathbb{Z}\times\mathbb{Z}\times\mathbb{Z}\setminus \{(0,0,0)\}$ and $\mathrm{gcd}(a,b,c) = 1$, then $|\mathrm{Lk}(h(R),C_x)| = |a|$, $|\mathrm{Lk}(h(R),C_y)| = |b|$, and $|\mathrm{Lk}(h(R),C_z)| = |c|$. 
\end{lemma} 

\begin{proof}
Let $D_x$, $D_y$, and $D_z$ be discs bounded by the link components $C_x$, $C_y$, and $C_z$ respectively. We view these as the shaded pink disc labelled with $F$ and a crown, the shaded green disc labelled with $G$ and a crown, and the white disc labelled with $B$ and $\underline{C}$, respectively, in \reffig{OctDecompositionS3}. 
Then $D_x$, $D_y$, and $D_z$ are Seifert surfaces for the link components $C_x$, $C_y$, and $C_z$ respectively. 

The homeomorphism $h^{-1}$ of \reflem{StdRodsDecomp} takes the shaded pink discs to the front and back faces of the cube in \reffig{OctDecompT3}; this is the $yz$-plane and its translate by $(1,0,0)$. The $(a,b,c)$-rod will intersect unit translates of the $yz$-plane in $\RR^3$ exactly $|a|$ times. All such intersections pass from the same side of the Seifert surface to the other. Therefore, $|\Lk(h(R),C_x)| = |a|$. 

Similarly, the $(a,b,c)$-rod meets the left and right faces of the unit cube exactly $|b|$ times; these are unit translates of the $xz$-plane in $\RR^3$, and the face shaded green in \reffig{OctDecompT3}, so $|\Lk(h(R),C_y)| = |b|$. And the $(a,b,c)$-rod meets top and bottom faces, labelled $B$ and $\underline{C}$ a total of $|c|$ times, so $|\Lk(h(R),C_z)| = |c|$. 
\end{proof} 

\section{Geometry of invariant cubic rod packings in crystallography} \label{Sec:RodPackingsCrystallography}

In the crystallography paper \cite{OKeeffeEtAl:CubicRodPackings}, O'Keeffe \emph{et al} find invariant cubic rod packings with axes along $(1,0,0)$ (there are two of them, denoted $^+\Pi$ and $\Pi^*$), axes along $(1,1,0)$ (there are none), and axes along $(1,1,1)$ (there are four, denoted $\Gamma$, $^+\Omega$, $^+\Sigma$ and $\Sigma^*$).  Here, a \emph{cubic} rod packing refers to a rod packing with crystallographic symmetry equal to a cubic space group. That is, it has the symmetry of a Euclidean cube, with right angles and equal side lengths. The term \emph{invariant} describes a cubic rod packing in which the axes of the rods are in directions determined by cubic symmetry, and there are no free parameters in the geometric locations of the rods (aside from scaling side lengths). 

In this section, we view invariant cubic rod packings as link complements in the 3-torus $\TT^3$ with cubic fundamental domain. This gives a $6$-component link for the $^+\Pi$ structure, a $3$-component link for the $\Pi^*$ structure and $4$-component links for the $\Gamma$, $^+\Omega$, $^+\Sigma$ structures.  We use the mapping between the complement of the Borromean rings in $\SS^3$ and the complement of the standard rods in $\TT^3$ to find hyperbolic structures on five of the six rod packings. The rod packings we consider are shown in \reffig{OKeeffePackings}. The one we omit, $\Sigma^*$, has eight components, and is the union of $^+\Sigma$ and its mirror image. 

\begin{figure}
\captionsetup[subfigure]{position=b}
\centering
\subcaptionbox{$^+\Pi$  \\
$\textcolor{blue}{R_1}: (0,0,0)+(0,0,1)t$ \\
$\textcolor{MyCyan}{R_2}: (\frac{5}{8},\frac{5}{8},0)+(0,0,1)t$ \\ 
$\textcolor{red}{R_3}: (0,\frac{1}{2},0)+(1,0,0)t$ \\
$\textcolor{magenta}{R_4}: (0,\frac{7}{8},\frac{1}{4})+(1,0,0)t$ \\
$\textcolor{ForestGreen}{R_5}: (\frac{1}{2},0,\frac{1}{2})+(0,1,0)t$ \\
$\textcolor{olive}{R_6}: (\frac{3}{4},0,\frac{1}{8})+(0,1,0)t$ \\
Volume $\approx 21.33$
\label{Fig:PlusPi} }
{\includegraphics{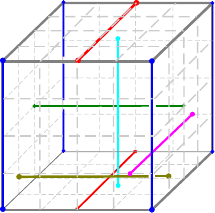}}
\quad \quad 
\subcaptionbox{$\Pi^*$ \\
$\textcolor{blue}{R_1}: (0,0,0)+(0,0,1)t$ \\
$\textcolor{red}{R_2}: (0,\frac{1}{2},0)+(1,0,0)t$ \\
$\textcolor{ForestGreen}{R_3}: (\frac{1}{2},0,\frac{1}{2})+(0,1,0)t$ 
Volume $\approx 7.33$ 
\label{Fig:PiStar}}
{\includegraphics{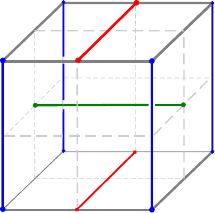}} \\ \vspace{4mm}
\subcaptionbox{$\Gamma$  \\
$\textcolor{MyCyan}{R_1}: (\frac{1}{8},0,\frac{1}{4})+(1,1,1)t$ \\ 
$\textcolor{magenta}{R_2}: (\frac{3}{8},\frac{3}{4},0)+(1,\text{-}1,1)t$ \\ 
$\textcolor{olive}{R_3}: (\frac{7}{8},\frac{1}{4},0)+(\text{-}1,\text{-}1,1)t$  \\
$\textcolor{orange}{R_4}: (\frac{3}{8},\frac{1}{4},0)+(\text{-}1,1,1)t$ 
Volume $\approx 24.36$ 
\label{Fig:Gamma} }
{\includegraphics{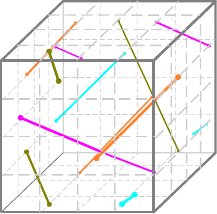}} 
\hspace{5mm} 
\subcaptionbox{$^+\Omega$ \\
$\textcolor{MyCyan}{R_1}: (\frac{1}{3},\frac{2}{3},0)+(1,1,1)t$ \\ 
$\textcolor{magenta}{R_2}: (\frac{2}{3},\frac{2}{3},0)+(1,\text{-}1,1)t$\\
$\textcolor{olive}{R_3}: (\frac{2}{3},\frac{1}{3},0)+(\text{-}1,\text{-}1,1)t$  \\
$\textcolor{orange}{R_4}: (\frac{1}{3},\frac{1}{3},0)+(\text{-}1,1,1)t$ 
Volume $\approx 24.09$ 
\label{Fig:PlusOmega} }
{\includegraphics{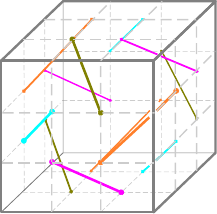}}
\hspace{5mm} 
\subcaptionbox{$^+\Sigma$\\
$\textcolor{MyCyan}{R_1}: (\frac{1}{3},\frac{2}{3},0)+(1,1,1)t$ \\ 
$\textcolor{magenta}{R_2}: (\frac{1}{6},\frac{2}{3},0)+(1,\text{-}1,1)t$\\
$\textcolor{olive}{R_3}: (\frac{2}{3},\frac{5}{6},0)+(\text{-}1,\text{-}1,1)t$  \\
$\textcolor{orange}{R_4}: (\frac{5}{6},\frac{5}{6},0)+(\text{-}1,1,1)t$ 
Volume $\approx 27.50$ 
\label{Fig:PlusSigma} }
{\includegraphics{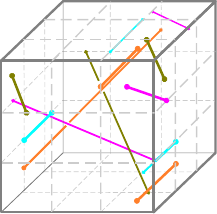}} 
\caption{Rod packings in \cite{OKeeffeEtAl:CubicRodPackings} as links in $\TT^3$. Parametric equations describe lifts of rods embedded in the $3$-torus. Hyperbolic volumes are also given, obtained from SnapPy \cite{SnapPy}. \label{Fig:OKeeffePackings}}
\end{figure}

\begin{theorem}\label{Thm:Crystallography}
Let $\mathcal{R}$ be any of the invariant cubic rod packings shown in \reffig{OKeeffePackings}. Then $\TT^3\setminus\mathcal{R}$ admits a complete hyperbolic structure. Moreover, the hyperbolic structures are all distinct. 
\end{theorem}

\begin{proof}
We study the cubic rod packings one at a time. Note that the $\Pi^*$ rod packing corresponds to the link complement $\TT^3\setminus (R_x \cup R_y \cup R_z)$, and so it is hyperbolic by \refthm{ThreeStandardRods}, built of regular ideal octahedra. 

For each of the other rod packings, we use \reflem{StdRodsDecomp} to turn the rod packing into a link diagram in $S^3$, then use tools to identify geometry available in this setting. In particular, the software SnapPy~\cite{SnapPy}, with the SageMath verification of geometry based on~\cite{hikmot}, verifies hyperbolicity in each case. 

We will work through $^+\Omega$ first. The octahedra split each rod $R_1$, $R_2$, $R_3$, and $R_4$ into four linear pieces running directly from one face of the octahedron to the other. We identify the faces at the endpoints of these linear pieces, then reproduce the pieces within the complement of the Borromean rings.

Consider first the rod $R_1$, which has a lift to the line $(\frac{1}{3},\frac{2}{3},0)+(1,1,1)t$ in $\RR^3$. This rod meets the white face $B$ of \reffig{OctDecompT3} on the bottom of the cube at the point $(\frac{1}{3},\frac{2}{3},0)$, then runs to $(\frac{2}{3},1,\frac{1}{3})$ on the green crown face on the right. A translate begins at $(\frac{2}{3},0,\frac{1}{3})$ on the green crown face on the left, meets the horizontal plane $z=\frac{1}{2}$ at the point $(\frac{5}{6},\frac{1}{6},\frac{1}{2})$ on face $A$, and continues to the point $(1,\frac{1}{3},\frac{2}{3})$ on the front pink face labelled $F$. Finally, a translate starts at the back pink face at $(0,\frac{1}{3},\frac{2}{3})$ and runs to $(\frac{1}{3},\frac{2}{3},1)$ on the top white face, identified to $B$. We draw the corresponding arcs in the octahedra between the faces identified 
and do a similar process for $R_2$, $R_3$, $R_4$. The rods in the octahedra are shown in \reffig{OctaDecompT3WithPlusOmega}.

\begin{figure}
\includegraphics[width=3cm]{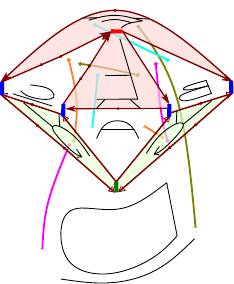}
\hspace{1cm}
\includegraphics[width=3cm]{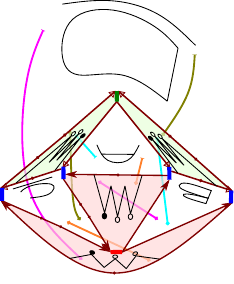}
\hspace{1cm}
\includegraphics[width=3.5cm]{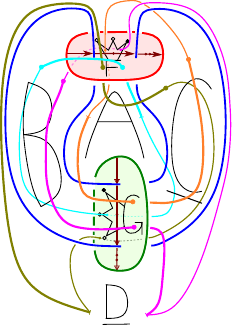}
\caption{Pieces of $^+\Omega$ in the octahedral decomposition. On the right is the link diagram with the Borromean rings in $\SS^3$.}
\label{Fig:OctaDecompT3WithPlusOmega}
\end{figure}   

The two ideal octahedra can be deformed and glued to obtain $h(^+\Omega)$ in $\SS^3 \setminus (C_x\cup C_y\cup C_z)$.  Note that when mapping from $\TT^3$ to $\SS^3$, we need to take care that the images of the rods have appropriate linking with each other, as well as with $C_x$, $C_y$, and $C_z$. This is ensured by mapping endpoints to appropriate points on the faces of the octahedra, and by observing that each arc cuts off a linear disc cobounded by the faces of the octahedra in $\TT^3$. Thus we ensure that the images in $\SS^3$ meet the faces of the octahedra in the decomposition
at appropriate intersection points, and the arc runs parallel to the faces between. The result is the link on the right of \reffig{OctaDecompT3WithPlusOmega}.

Now we can input the $7$-component link on the right of \reffig{OctaDecompT3WithPlusOmega} into SnapPy. By \refprop{010101DehnFilling}, we apply $((0,1),(0,1),(0,1))$-Dehn filling on the Borromean rings to get back $\TT^3\setminus (^+\Omega)$. 
It admits a decomposition into positively oriented ideal tetrahedra, and using the verify hyperbolicity method, we check that the result is hyperbolic.  SnapPy computes the hyperbolic volume, which can be verified in Sage up to high precision. 

Using similar methods, we find that $\TT^3\setminus (^+\Pi)$, $\TT^3\setminus (\Gamma)$, and $\TT^3\setminus (^+\Sigma)$ admit complete hyperbolic structure.

The PLink diagrams for these links are shown in \reffig{PLinkDiagrams}.
\begin{figure}
  \includegraphics[width=1.5in]{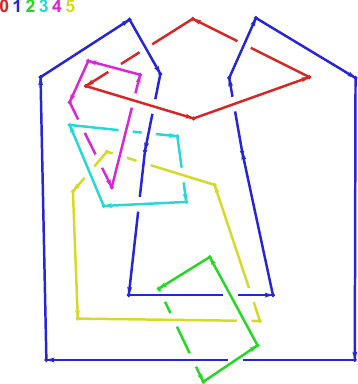}
  \hspace{.2in}
  \includegraphics[width=1.5in]{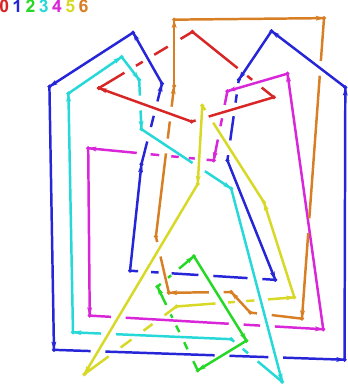}

  \includegraphics[width=1.5in]{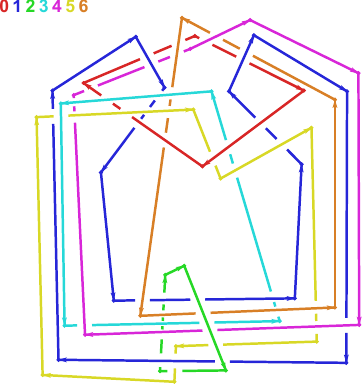}
  \hspace{.2in}
  \includegraphics[width=1.5in]{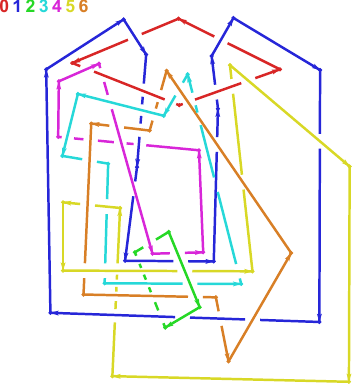}
  \caption{Shown are the union of the Borromean rings (red, blue, green) and the cubic rod packings: $h(^+\Pi)$ (top left), $h(\Gamma)$ (top right), $h(^+\Omega)$ (lower left), and $h(^+\Sigma)$ (lower right)}
  \label{Fig:PLinkDiagrams}
\end{figure}
The three links with four components are verified to have distinct hyperbolic structures, as their volumes are distinct. The other two have three and six link components, respectively, and so are also distinct. 
\end{proof}

\begin{remark}
The invariant rod packing $\Pi^*$ has primitive cubic lattice symmetry, meaning its full translational group is generated by unit translations in the $x$, $y$, and $z$-directions. The other five rod packings have body-centred cubic lattice symmetry, meaning each has additional translational symmetry in the $(\frac{1}{2}, \frac{1}{2}, \frac{1}{2})$-direction.
If we quotient the $^+\Pi$ rod packing by its full translational symmetry, we obtain a $3$-component link in $\TT^3$ with fundamental domain in a parallelepiped spanned by $(1,0,0)$, $(0,1,0)$ and $(\frac{1}{2},\frac{1}{2},\frac{1}{2})$.   The $6$-component link complement of $^+\Pi$ is a cover of this. Hyperbolic structures of $\Pi^*$ and this quotient of $^+\Pi$ can still be seen to be distinct, for example by hyperbolic volume. 
Similarly, quotients of the complements of the $\Gamma$, $^+\Omega$, and $^+\Sigma$ rod packings by the additional body-centred symmetry give $4$-component links. Because hyperbolic volumes are distinct before taking the quotient, they remain distinct after.
\end{remark}

\section{One and two arbitrary rods in the $3$-torus} 
\label{Sec:OneTwoRods}

As in the previous section, we can consider examples of rod packings individually to determine the geometric type, using the homeomorphism of \reflem{StdRodsDecomp}. For example, this may be useful for other well-known examples arising from crystallography. However, we would like to be able to make more general statements about the geometry of the complement of infinite families of rod packings. We begin that process here, considering first the simplest cases, namely one and two arbitrary rods in the $3$-torus.   

\begin{theorem} \label{Thm:MainSingleRod_TwoRods}
Let $R_1$ and $R_2$ be rod-shaped circles embedded in $\TT^3$.
\begin{itemize}
\item In the case of one rod, the $3$-manifold $\TT^3\setminus R_1$ is Seifert fibred. 

\item If $R_1$ and $R_2$ lift to be parallel to linearly dependent vectors in $\RR^3$, then $\TT^3\setminus(R_1\cup R_2)$ is Seifert fibred.
\item If $R_1$ and $R_2$ lift to linearly independent vectors, then $\TT^3\setminus(R_1\cup R_2)$ is toroidal.
\end{itemize}
In any case, when there are one or two rods, the complement is not hyperbolic. 
\end{theorem} 

\begin{proof} 
In the case of a single rod $R_1$, we can foliate the $3$-manifold $\TT^3$ by rod-shaped circles parallel to $R_1$. Therefore, $\TT^3\setminus R_1$ is also foliated by such circles. It follows that $\TT^3\setminus R_1$ is a Seifert fibred space. 

In the case of two rods $R_1$ and $R_2$, consider the universal cover $\mathbb{R}^3$ of $\TT^3$. The lifts of $R_1$ and $R_2$ in $\mathbb{R}^3$ are straight lines with vector directions that span a straight line or a plane. If the two vectors span a straight line, then all lines parallel to that line give a fibration of $\RR^3$. Projecting to $\TT^3$, they give a fibring by circles.  Since two of these circles are $R_1$ and $R_2$, the complement of these rod-shaped circles is a Seifert fibred space. 

Now suppose the vector directions are linearly independent, spanning a plane in $\RR^3$. 
Since there are only finitely many lifts of $R_1$ and $R_2$ within each unit cube in $\mathbb{R}^3$, we can always find a plane $\Pi$ parallel to both the lifts of $R_1$ and the lifts of $R_2$ such that $\Pi$ is disjoint from all the lifts of rods. 

Because $R_1$ and $R_2$ project to closed curves in $\TT^3$, the image $\mathcal{P}(\Pi)$ in the $3$-torus under the covering map is a plane torus $\SS^1 \times \SS^1$. Note that $\calP(\Pi)$ is $\pi_1$-injective in $M=\TT^3\setminus (R_1\cup R_2)$, and thus incompressible. It cannot be boundary-parallel in $M$ because it would lift to the boundary of a tubular neighbourhood of lines lifting $R_1$ or $R_2$ in $\RR^3$. But $\Pi$ is parallel to these rods. Thus $\calP(\Pi)$ is an essential torus. 

By Thurston's Hyperbolization Theorem, the $3$-manifold $\TT^3\setminus (R_1 \cup R_2)$ cannot admit a complete hyperbolic structure. 
\end{proof}

\section{The standard rods plus another rod} \label{Sec:StdRodsPlusRod} 

In light of the results of the previous section, the next most interesting general case might be to consider the geometric structures of the complements of three arbitrary rod-shaped circles in $\TT^3$. Indeed, experimentally, such link complements are frequently hyperbolic. However, proving hyperbolicity when an arbitrary rod is added to two arbitrary rods seems challenging. Instead, the next simplest case seems to be the case of adding an arbitrary rod to the standard rods, and considering the geometry of the complement of those four rods. The presence of the standard rods allows us to transfer geometric problems to and from the Borromean rings complement in $\SS^3$, using tools in that setting. As opposed to the case of one or two rod-shaped circles, many hyperbolic examples arise. This section characterises exactly when such link complements are hyperbolic. The goal is to prove the following.

\begin{theorem} \label{Thm:MainCharacterizeGeometryR} 
Let $a$, $b$, $c$ be integers such that $\mathrm{gcd}(a,b,c)=1$, and let $R$ be an $(a,b,c)$-rod in $\TT^3$. Let $R_x$, $R_y$, $R_z$ denote the three standard rods in $\TT^3$. 
Then the 3-manifold $\TT^3\setminus(R_x \cup R_y \cup R_z \cup R)$ admits a complete hyperbolic structure if and only if $(a,b,c) \notin \{(\pm 1, 0, 0), (0,\pm 1, 0), (0, 0, \pm 1)\}$. 
\end{theorem}

We will use Thurston's hyperblization theorem~\cite{Thurston:3MKleinianGroupsHG}, which states that a 3-manifold that is the interior of a compact manifold with torus boundary is hyperbolic if and only if it is irreducible, boundary-irreducible, atoroidal, and anannular.
We first prove the non-hyperbolic result. 
 
\begin{proposition} \label{Prop:RParallelStdRod} 
For any rod $R$ that is parallel to the vector $(1,0,0)$, $(0,1,0)$ or $(0,0,1)$, 
the manifold $\TT^3\setminus (R_x\cup R_y\cup R_z\cup R)$ does not admit any complete hyperbolic structure. 
\end{proposition} 

\begin{proof}
By the symmetry of the three standard rods, it suffices to show one of the three cases. Without loss of generality, let $R$ be a rod that is parallel to the vector $(0,1,0)$ and is disjoint from the other three standard rods.

Note that there is an annulus $A$ embedded in $M:=\TT^3\setminus N(R_x\cup R_y\cup R_z\cup R)$ with one boundary component on $R$ and one on $R_y$. We show this is essential, so by Thurston's hyperbolization theorem, the manifold $\TT^3\setminus (R_x\cup R_y\cup R_z\cup R)$ cannot be hyperbolic.

Observe $A$ is not boundary-parallel because its two boundary components lie on distinct link components. It cannot be boundary-compressible, because if $D$ is a disc with $\bdy D$ consisting of an arc on $A$ and an arc on $\bdy M$, then the arc on $A$ must cut off a disc in $A$. Finally, it cannot be compressible because the core of $A$ represents a generator of $\pi_1(\TT^3)=\ZZ\times\ZZ\times\ZZ$. 
\end{proof}

\subsection{Irreducibility and boundary-irreducibility}  \label{Sec:IrredAndBdryIrred}

In this section, we prove that a rod packing as in \refthm{MainCharacterizeGeometryR} is irreducible and boundary-irreducible when $(a,b,c) \notin \{(\pm 1, 0, 0), (0,\pm 1, 0), (0, 0, \pm 1)\}$. 

\begin{proposition} \label{Prop:IrredABdryIrred}  
Let $a$, $b$, $c$ be integers such that $\mathrm{gcd}(a,b,c)=1$, and let $R$ be an $(a,b,c)$-rod in $\TT^3$. Further suppose that $(a,b,c)$ is not in $\{(\pm 1, 0, 0), (0,\pm 1, 0), (0, 0, \pm 1)\}$. Then the $3$-manifold 
$M=\TT^3\setminus N(R_x \cup R_y \cup R_z \cup R)$ is irreducible and boundary-irreducible. 
\end{proposition}  

\begin{proof} 
Without loss of generality, assume $a\neq 0$. By \reflem{LinkingNum}, the linking number of $h(R)$ and $C_x$ satisfies $|\mathrm{Lk}(h(R), C_x)| = |a| \neq 0$, where $h$ is the homeomorphism of \reflem{StdRodsDecomp}. 

Let $S$ be any $2$-sphere in $\SS^3\setminus(C_x\cup C_y\cup C_z \cup h(R))$. Since the complement of the Borromean rings in $\SS^3$ is hyperbolic and thus irreducible, $C_x\cup C_y\cup C_z$  will lie in a $3$-ball $B$ bounded by $S$. Since $|\mathrm{Lk}(h(R), C_x)| = |a| \neq 0$, $h(R)$ will lie in the same $3$-ball $B$. Hence, $S$ bounds a $3$-ball $\SS^3\setminus B$ in $\SS^3\setminus(C_x\cup C_y\cup C_z \cup h(R))$, so $\TT^3\setminus (R_x \cup R_y \cup R_z \cup R) \cong \SS^3\setminus (C_x\cup C_y\cup C_z\cup h(R))$ is irreducible. 

Now suppose, by way of contradiction, $M$ is $\bdy$-reducible. Then there exists a properly embedded disc $D$ in $M$ such that $\partial D$ does not bound a disc in $\partial M$. Then $\partial D$ is some nontrivial simple closed curve in a torus component of $\partial M$. Note that $\partial D$ does not lie on $\bdy \overline{N}(R_x\cup R_y\cup R_z)$, else $h(D)$ would be a properly embedded disc in the Borromean rings complement that does not bound a disc in the boundary, contradicting the hyperbolicity of that link. Therefore, $\bdy D$ lies on $\bdy \overline{N}(R)$.

\textbf{Case BI(i)}: Suppose $\bdy D$ is a (1,0)-curve of $\bdy \overline{N}(R)$, i.e.\ a meridian for the rod $R$.
The homeomorphism $h$ sends the meridional disc $D_m$ of $\overline{N}(R)\subset \TT^3$ to a meridional disc $h(D_m)$ of $h(\overline{N}(R))$ bounded by $h(\partial D)$. Since the interior of $D$ lies in $M$, we have $h(D)$ lies in $h(M)$, with interior disjoint from $h(D_m)$. Therefore, $h(D)\cup h(D_m)$ is a $2$-sphere in $\SS^3$ that intersects the loop $h(R)$ transversely only once. This is impossible because a $2$-sphere is separating in $\SS^3$, but the loop $h(R)$ is not cut into two components by the sphere. 

\textbf{Case BI(ii)}: Suppose $\partial D$ is a $(p,q)$-curve of $\partial\overline{N}(R)$ with $q\neq 0$.  Note that the homotopy class $[R]$ of the $(a,b,c)$-rod $R$ is not an identity element in $\pi_1(\TT^3)$.  Since a $(p,q)$-curve is homotopic to a $(0,q)$-curve in the solid torus $N(R)$, the homotopy class $[\bdy D]$ equals $[R^q]$, which is the homotopy class of the product path $R^q$ in $\pi_1(\TT^3)$.  Thus, we have 
\[ [R]^q = [R^q] = [\partial D] = \textrm{Id}.\]
The last equality follows from the fact that $D$ is a disc embedded in $M \subset \TT^3$, hence $\bdy D$ is homotopically trivial. But the group $\pi_1(\TT^3)$ has no non-identity element with finite order. Thus $M$ is $\bdy$-irreducible. 
\end{proof}

\subsection{Non-existence of essential tori} \label{Sec:AT}

We now turn our attention to essential tori. Throughout this subsection, denote by $M$ and $M'$ the following compact 3-manifolds: 
\begin{align}
  M & \coloneqq \TT^3\setminus N(R_x \cup R_y \cup R_z \cup R), \label{Eqn:M}\\
  M'& \coloneqq \SS^3\setminus N(C_x\cup C_y\cup C_z\cup h(R)).\label{Eqn:Mprime}
\end{align}
  
\begin{lemma} \label{Lem:TNotBdryParallel} 
If $M$ contains an embedded essential torus $T_e$, then the torus $h(T_e)$ in $M'$ cannot be boundary-parallel after $(1,0)$-Dehn-filling $h(R)$.
\end{lemma}

\begin{proof}
Suppose not. Since the embedded torus $h(T_e)$ becomes boundary-parallel after $(1,0)$-Dehn-filling $h(R)$, the torus $h(T_e)$ bounds a solid torus $V$ in $\SS^3$ that contains a link component of the Borromean rings as a core curve, and does not contain the other two link components. Without loss of generality, say $V$ contains $C_x$ and is disjoint from $C_y$ and $C_z$. Because $T_e$ is not boundary-parallel in $M'$, $V$ also contains $h(R)$. 

Note that $C_x$ has zero linking number with $C_y$ and $C_z$ respectively and $h(R)$ is contained in $V$, which we may view as $N(C_x)$.  We may take a link diagram such that $N(C_x)$ and $C_y$ do not form any crossings, for example as in \reffig{Homeomorphic}, right. 

Thus, there exists a link diagram such that the link components $h(R)\subset N(C_x)$ and $C_y$ do not form any crossings. Hence, we have $\Lk(h(R),C_y)=0$.
Using a similar argument, $\Lk(h(R),C_z)=0$. By \reflem{LinkingNum}, $R$ must be a $(1,0,0)$-rod. This contradicts our assumption for $R$. 
\end{proof}

\begin{lemma}\label{Lem:BoundSolidTorus}
If $T_e$ is an essential embedded torus in $M$, but $h(T_e)$ is compressible after $(1,0)$ Dehn filling $h(R)$, then $h(T_e)$ bounds a solid torus in $\SS^3$ that contains $h(R)$. 
\end{lemma}

\begin{proof}
Suppose $h(T_e)$, embedded in $M'$, is compressible in $\SS^3\setminus N(C_x\cup C_y\cup C_z)$ with compressing disc $D$. The disc $D$ intersects $h(R)$ because otherwise, the disc $h^{-1}(D)$ would be a compressing disc for $T_e$ in $M$, contradicting the assumption that $T_e$ is an essential torus in $M$. 

Note that the embedded torus $h(T_e)$ bounds a solid torus $V$ in $\SS^3$.
If $D\not\subset V$, then $V$ is an unknotted solid torus, and $D$ would be a subset of a second solid torus bounded on the other side of $h(T_e)$. Therefore, it suffices to consider $D\subset V$. 
Since $D \cap h(R) \neq \emptyset$, the interior of $V$ contains $h(R)$.
\end{proof}

In Lemmas~\ref{Lem:Vnot1or2}, \ref{Lem:VnotAll3} and~\ref{Lem:Vnot0}, suppose that $T_e$ is an essential torus that is compressible in the link complement $\TT^3\setminus N(R_x\cup R_y\cup R_z)$. Let $V$ denote the solid torus in $\SS^3$ obtained from \reflem{BoundSolidTorus}, and let $D$ be a compressing disc as in the proof of that lemma.  Recall that $D$ intersects $h(R)$ but none of the links of the Borromean rings. Thus $\bdy D$ is homotopic to a $(1,0)$-curve on $\bdy V=h(T_e)$. 

\begin{lemma} \label{Lem:Vnot1or2} 
The solid torus $V$ cannot contain exactly one or exactly two link components of the Borromean rings. 
\end{lemma} 

\begin{proof}
Suppose not. Surger the torus $T_e$ along the embedded meridional disc $D$. 
This gives an embedded $2$-sphere $S$ in $\SS^3\setminus N(C_x\cup C_y\cup C_z)$ with exactly one or two components of $C_x$, $C_y$, $C_z$ on one side, and the other(s) on the opposite side. Thus $S$ is a 2-sphere that does not bound a $3$-ball on either side. This contradicts the fact that the complement of Borromean rings in $\SS^3$ is irreducible. 
\end{proof}

\begin{lemma} \label{Lem:VnotAll3} 
The solid torus $V$ cannot contain all three link components of the Borromean rings. 
\end{lemma} 

\begin{proof}  
Suppose the solid torus $V$ contains all of $C_x$, $C_y$, and $C_z$. By \reflem{BoundSolidTorus}, it also contains $h(R)$. Since $\bdy V =h(T_e)$ is essential in $M'$, but $V$ contains all four link components, the solid torus $V$ must form a non-trivial knot in $\SS^3$, and $\SS^3\setminus V$ is a nontrivial knot complement. 
Surgering $h(T_e)$ along $D$ in the Borromean rings complement gives a sphere in $\SS^3$ that must bound a ball $B\subset \SS^3$ disjoint from $C_x$, $C_y$, and $C_z$. It follows that $D$ lies inside $B$, and $h(R)$ runs through $B$. By isotoping slightly, we may take $h(T_e)$ to lie inside of the ball $B$, and hence all of the knot complement $\SS^3-V$ lies in $B$. 

Now consider $h^{-1}(B)$ in $\TT^3$. This is a ball that contains $T_e$ and a submanifold $h^{-1}(\SS^3-V)$, and must also meet the rod $R$. By compactness, it meets $R$ finitely many times. The ball $h^{-1}(B)$ must lift to a ball $\tilde{B}$ in the universal cover $\RR^3$, disjoint from lifts of the standard rods, but intersecting finitely many lifts of $R$. Observe that a lift $\tilde{T}_e$ of $T_e$ must lie inside $\tilde{B}$ and be disjoint from all lifts of rods. The disc $h^{-1}(D)$ lifts to a disc meeting lifts of $h(R)$.

Then the torus $\tilde{T}_e$ has to be a knotted torus in the ball $\tilde{B}$, but the finite set of lifts of $R$ that meet $B$ is a union of parallel straight lines. These cannot intersect each meridional disc of a knotted tube in $\RR^3$.  A compressing disc for $\tilde{T}_e$ would thus exist, and project to a compressing disc for $T_e$, contradicting the assumption that $T_e$ is incompressible.
\end{proof}

\begin{lemma} \label{Lem:Vnot0}
The solid torus $V$ cannot be disjoint from all three link components of the Borromean rings. 
\end{lemma} 

\begin{proof}
Suppose on the contrary that the solid torus $V$ does not contain any link component of the Borromean rings. Then $V$ contains only $h(R)$. 

Note that $[h(R)]$ represents a homotopy class in $\pi_1(V) \cong \pi_1(\SS^1) \cong \mathbb{Z}$ and that any homotopy of $h(R)$ within the solid torus $V$ does not change the linking number between $h(R)$ and each of the three link components of the Borromean rings. By \reflem{LinkingNum}, these are $\pm a$, $\pm b$, $\pm c$ respectively. Since the solid torus $V$ contains $h(R)$, it must also link $C_x$, $C_y$, $C_z$. 

\textbf{Case 1:} Suppose the homotopy class $[h(R)]$ corresponds to an integer $n$ in $\mathbb{Z}\setminus\{1,0,-1\}$ via the  group isomorphism between $\pi_1(V)$ and $\mathbb{Z}$. Then $|a|=|n\Lk(V,C_x)|$, $|b|=|n\Lk(V,C_y)|$, and $|c|=|n\Lk(V,C_z)|$. But this contradicts the fact that $a$, $b$, $c$ are relatively prime. 

\textbf{Case 2:} Suppose the homotopy class $[h(R)]$ corresponds to the trivial element $0 \in \mathbb{Z}$.  Then $h(R)$ is homotopic to a loop that bounds a disc.  The embedded circle $h(R)$ would not be linked to any other link components, contradicting \reflem{LinkingNum}.  

\textbf{Case 3:} Suppose the homotopy class $[h(R)]$ corresponds to an integer $1 \textrm{ or } -1 \in \mathbb{Z}$. The rod parallel to $(a,b,c)$ hits each face of the octahedra in the cube in the same direction. Hence $h(R)$ hits each disc $D_x$, $D_y$, $D_z$ bounded by $C_x$, $C_y$, $C_z$ respectively transversely in the same direction. 
Thus it must meet each meridional disc $V\cap D_i$ exactly once in the same direction. Since the rod is monotonic between faces, there is no local knotting within the octahedron. Thus $h(R)$ meets each meridional disc of $V$ once. 
But then $\partial V = h(T_e)$ is boundary-parallel in $M'$, a contradiction.
\end{proof}

\begin{proposition} \label{Prop:TNotCompressible}
If $M$ admits an embedded essential torus $T_e$, then the torus $h(T_e)$ in $M'$ cannot be compressible after $(1,0)$-Dehn-filling $h(R)$.  
\end{proposition} 

\begin{proof}
Suppose not. By \reflem{BoundSolidTorus}, $h(T_e)$ bounds a solid torus in $\SS^3$ that contains $h(R)$. By \reflem{Vnot1or2}, it must contain either all or none of the link components of the Borromean rings. By \reflem{VnotAll3}, it cannot contain all three. By \reflem{Vnot0} it cannot contain none. Thus no such torus exists.
\end{proof}

\begin{theorem} \label{Thm:Atoroidal}  
Let $a$, $b$, $c$ be integers such that $\mathrm{gcd}(a,b,c)=1$. If the rod $R$ is parallel to the vector $(a,b,c) \in (\mathbb{Z}\times\mathbb{Z}\times\mathbb{Z}) \setminus \{(\pm{1},0,0),(0,\pm{1},0),(0,0,\pm{1})\}$, then the $3$-manifold 
$\TT^3\setminus(R_x \cup R_y \cup R_z \cup R)$ is atoroidal. 
\end{theorem}  

\begin{proof}
Suppose there were an essential torus $T_e$ in $M\subset \TT^3$. Then $h(T_e)$ is an essential torus in $M'\subset \SS^3$. As the complement of the Borromean rings in $\SS^3$ admits a complete hyperbolic structure, the essential torus $h(T_e)$ is not essential after $(1,0)$-Dehn filling $h(R)$.  Therefore, $h(T_e)$ either becomes boundary-parallel or compressible after the Dehn filling. By \reflem{TNotBdryParallel}, it cannot be boundary-parallel. By \refprop{TNotCompressible}, it cannot be compressible. Hence, $M$ is atoroidal. 
\end{proof}

\subsection{Non-existence of essential annuli} \label{Sec:AA}

\begin{theorem} \label{Thm:Anannular}  
Let $a$, $b$, $c$ be integers such that $\mathrm{gcd}(a,b,c)=1$. If the rod $R$ is parallel to the vector $(a,b,c) \in (\mathbb{Z}\times\mathbb{Z}\times\mathbb{Z}) \setminus \{(\pm{1},0,0),(0,\pm{1},0),(0,0,\pm{1})\}$, then the $3$-manifold 
$\TT^3\setminus(R_x \cup R_y \cup R_z \cup R)$ is anannular. 
\end{theorem}  

To prove \refthm{Anannular}, we will first need the following more general result about curves on neighbourhoods of rod-shaped circles in $\TT^3$. 

\begin{lemma} \label{Lem:NonHmtpCurves} 
Suppose $R_1$ and $R_2$ are two rod-shaped circles in $\TT^3$ that are not parallel to one another. Then a $(p,q)$-curve around $R_1$ is not homotopic to an $(r,s)$-curve around $R_2$ in $\TT^3\setminus (R_1\cup R_2)$, where $p,q,r,s\in \ZZ$ are such that $(p,q), (r,s)\neq (0,0)$. 
\end{lemma}

\begin{proof}
Since $R_1$ and $R_2$ are simple closed curves in $\TT^3$, there exist $a$, $b$, $c$, and $\ell$, $m$, $n\in\mathbb{Z}$ with $\mathrm{gcd}(a,b,c) = \mathrm{gcd}(\ell,m,n) = 1$ such that $R_1$ is parallel to $(a,b,c)$ and $R_2$ is parallel to $(\ell,m,n)$. 

\textbf{Case 1:} $q=0$ and $s=0$. 
If a $(1,0)$-curve $\gamma_1$ around $R_1$ is homotopic to a $(1,0)$-curve $\gamma_2$ around $R_2$ in $\TT^3\setminus (R_1\cup R_2)$, then $\gamma_1$ is homotopic to $\gamma_2$ in $\TT^3\setminus R_1$. Thus $\gamma_1$ is homotopically trivial in $\TT^3\setminus R_1$. Contradiction. 

\textbf{Case 2:} $q\neq 0$ or $s \neq 0$. Without loss of generality, assume $q\neq 0$.
Observe that the $(p,q)$-curve around $R_1$ is homotopic to the product path $(R_1)^q$ in $\TT^3$ and the $(r,s)$-curve around $R_2$ is homotopic to $(R_2)^s$ in $\TT^3$. If these are homotopic, $[R_1]^q = [R_2]^s$ in the fundamental group of $\TT^3$. But the two rods are not parallel, so $[(R_1)]^q$ and $[(R_2)]^s$ are distinct elements in the fundamental group of $\TT^3$. 
\end{proof}

As in the previous subsection, we will denote by $M$ and $M'$ the compact 3-manifolds of equations~\eqref{Eqn:M} and~\eqref{Eqn:Mprime}.
We will prove \refthm{Anannular} by contradiction. So suppose it does not hold. Then there is an essential annulus in $M$, which will be denoted by $A_e$. Denote the two connected components of  $\partial A_e$ by $(\partial A_e)_1$ and $(\partial A_e)_2$ respectively. 

\begin{lemma} \label{Lem:BothNontrivialInTheSameTorus}
Both boundary components $(\partial A_e)_1$ and $(\partial A_e)_2$ of the essential annulus must be nontrivial loops that lie in the same torus boundary component of $M$. 
\end{lemma} 

\begin{proof}
First note that $(\partial A_e)_1$ and $(\partial A_e)_2$ are nontrivial loops in $\partial M$, for if $(\bdy A_e)_j$ bounds a disc on $\bdy M$, then slightly pushing the disc off of $\bdy M$ gives a compressing disc for $A_e$, contradicting the assumption that $A_e$ is essential. 

The annulus $A_e$ provides a homotopy between $(\bdy A_e)_1$ and $(\bdy A_e)_2$ in $M$. Since $(\partial A_e)_1$ and $(\partial A_e)_2$ are nontrivial loops that are homotopic, \reflem{NonHmtpCurves} implies they cannot lie on distinct rods. \end{proof}

Let $(\bdy M)_i$ denote the torus boundary component of $M$ that contains the two boundary components of $A_e$. 
Because $(\bdy A_e)_1$ and $(\bdy A_e)_2$ are disjoint nontrivial curves on the same torus, they divide $(\bdy M)_i$ into two annuli. Denote by $A^*$ and $A^{**}$ the two annuli in $(\bdy M)_i \setminus ((\bdy A_e)_1 \cup (\bdy A_e)_2)$. Also, denote by $R_i\in \{R_x, R_y, R_z, R\}$ and $L_i\in\{C_x, C_y, C_z, h(R)\}$ the link components in $\TT^3$ and $\SS^3$ respectively such that $h(\bdy \overline{N}(R_i)) = h((\bdy M)_i) = \bdy \overline{N}(L_i)$. 

\begin{lemma} \label{Lem:NotBoundaryParallel}
The tori $(A_e\cup A^*)$ and $(A_e\cup A^{**})$ are not boundary-parallel.
\end{lemma} 

\begin{proof}
Suppose not. Without loss of generality, assume $(A_e\cup A^*)$ is boundary-parallel in $M$; the case of $(A_e\cup A^{**})$ is similar. Then $(A_e\cup A^*)$ can be isotoped to $\bdy \overline{N}(R_k)$ for some link component $R_k \in \{R_x, R_y, R_z, R\}$. 

\textbf{Case 1}: Suppose $R_k = R_i$. Then $(A_e\cup A^*)$ can be isotoped to $\bdy \overline{N}(R_i)$ in $M$. Thus, the annulus $A_e$ can also be isotoped to $\bdy \overline{N}(L_i)$ in $M$, contradicting the assumption that $A_e$ is essential. 

\textbf{Case 2}: Suppose $R_k \neq R_i$. Note that $\bdy A_e \subset \bdy \overline{N}(R_i)$. The isotopy taking $(A_e\cup A^*)$ to $\bdy \overline{N}(R_k)$ restricts to take $(\partial A_e)_1$ to a curve on $\bdy \overline{N}(R_k)$. Observe it must be a nontrivial curve, by incompressibility of $A_e$. Additionally, by \reflem{BothNontrivialInTheSameTorus}, $(\bdy A_e)_1$ is a nontrivial $(p,q)$-curve in $\bdy \overline{N}(R_i)$. Then $(\bdy A_e)_1$ is homotopic to a nontrivial loop in $\bdy \overline{N}(R_k)$, contradicting \reflem{NonHmtpCurves}. 
\end{proof} 

\begin{proof} [Proof of \refthm{Anannular}]
Since $h$ induces a homeomorphism from $M$ to $M'$, by \refthm{Atoroidal}, the tori $h(A_e\cup A^*)$ and $h(A_e\cup A^{**})$ are not essential in $M'$. Thus, $h(A_e\cup A^*)$ and $h(A_e\cup A^{**})$ are compressible or boundary-parallel. They cannot be boundary-parallel, by \reflem{NotBoundaryParallel}. Therefore, each is compressible. 

Consider first $h(A_e\cup A^*)$.
Let $D$ be a compressing disc for the torus $h(A_e\cup A^{*})$ in $M'$. Note $\bdy D$ cannot lie completely in $h(A_e)$, else it gives a compressing disc for $h(A_e)$, which is incompressible by assumption. Thus $\bdy D$ meets $h(A^*)$. 

Now surger $h(A_e\cup A^*)$ along $D$ to obtain a sphere $S$. By \refprop{IrredABdryIrred}, the sphere $S$ must bound a $3$-ball in $M'$. This ball must be disjoint from the link components $C_x$, $C_y$, $C_z$, and $h(R)$. This implies it cannot contain $D$, since $\bdy D$ meets $h(A^*)$, which forms part of $\bdy \overline{N}(L_i)$. Then undoing the surgery turns the ball into a solid torus $V^*$ bounded by $h(A_e\cup A^*)$, disjoint from the link components.

Applying similar argument to $h(A_e\cup A^{**})$, we find $h(A_e\cup A^{**})$ also bounds a solid torus $V^{**}$ disjoint from all the link components.

But now, $V^*\cup V^{**}$ is a submanifold of $M'$ with boundary $h(A^*\cup A^{**}) = \bdy \overline{N}(L_i)$ that contains none of the link components $C_x$, $C_y$, $C_z$, or $h(R)$. But $N(L_i)$ contains only one of these link components. This is a contradiction. Hence, the $3$-manifold 
$\TT^3\setminus(R_x \cup R_y \cup R_z \cup R)$ is anannular.
\end{proof}

\subsection{A family of hyperbolic links in the $3$-torus} \label{Sec:MainResult}

\begin{proof}[Proof of \refthm{MainCharacterizeGeometryR}]
Suppose $(a,b,c) \in \{(\pm{1},0,0),(0,\pm{1},0),(0,0,\pm{1})\}$, and $R$ is a rod parallel to the vector $(a,b,c)$ in $\TT^3$. By \refprop{RParallelStdRod}, $\TT^3\setminus (R_x \cup R_y \cup R_z \cup R)$ does not admit any complete hyperbolic structure. 

Now suppose $(a,b,c)\not\in \{(\pm{1},0,0),(0,\pm{1},0),(0,0,\pm{1})\}$. 
Consider the compact manifold
$ M\coloneqq \TT^3\setminus N(R_x \cup R_y \cup R_z \cup R)$.
By \refprop{IrredABdryIrred}, $M$ is irreducible and $\partial$-irreducible. 
By \refthm{Atoroidal}, $M$ is atoroidal. By \refthm{Anannular}, $M$ is anannular.
Therefore Thurston's hyperbolisation theorem implies that the interior of $M$ admits a complete hyperbolic structure. 
\end{proof}

\bibliographystyle{amsplain}  
\bibliography{references}

\end{document}